\documentclass[reqno]{amsart}
\usepackage{amsmath,amsthm,amssymb}

\makeatletter
    
    \@addtoreset{equation}{section}
  \makeatother

\newtheorem{definition}{Definition}[section]
\newtheorem{proposition}[definition]{Proposition}

\newtheorem{lemma}[definition]{Lemma}

\theoremstyle{definition}

\newcommand{\ep}{\varepsilon}

\newcommand{\R}{\mathbb{R}}

\newcommand{\N}{\mathbb{N}}
\newcommand{\pa}{\partial}
\newcommand{\lr}[1]{\langle{#1}\rangle}

\pagebreak

\title[Wave equation with space-dependent damping]{
Remarks on an elliptic problem 
arising in weighted energy estimates 
for wave equations with 
space-dependent damping term in an exterior domain}

\author{Motohiro Sobajima}
\address[M. Sobajima]{Department of Mathematics, 
Faculty of Science and Technology, Tokyo University of Science,  
2641 Yamazaki, Noda-shi, Chiba-ken 278-8510, Japan}
\email{msobajima1984@gmail.com}

\author{Yuta Wakasugi}
\address[Y. Wakasugi]{Graduate School of Mathematics, Nagoya University,
Furocho, Chikusaku, Nagoya 464-8602 Japan}
\email{yuta.wakasugi@math.nagoya-u.ac.jp}

\begin{document}
\begin{abstract}
This paper is concerned with 
weighted energy estimates and diffusion phenomena 
for the initial-boundary problem of the wave equation 
with space-dependent damping term in an exterior domain. 
In this analysis, an elliptic problem was introduced 
by Todorova and Yordanov. 
This attempt was quite useful when the coefficient of 
the damping term is radially symmetric.
In this paper, by modifying their elliptic problem, 
we establish weighted energy estimates and diffusion 
phenomena even when the coefficient of the damping 
term is not radially symmetric.
\end{abstract}

\maketitle

\section{Introduction}
Let $N\geq 2$. We consider 
the wave equation with space-dependent damping term 
in an exterior domain $\Omega\subset \R^N$ 
with a smooth boundary:
\begin{align}
\label{dw}
	\left\{\begin{array}{ll}
	u_{tt}-\Delta u + a(x)u_t = 0,& x\in \Omega,\ t>0,
\\
	u(x,t)=0,&x\in \partial \Omega, \ t>0,
\\
	(u,u_t)(x,0) = (u_0, u_1)(x),&x\in \Omega,
	\end{array}\right.
\end{align}
where we denote by $\Delta$ the usual Laplacian in $\R^N$ 
and 
by $u_t$ and $u_{tt}$ the first and second derivative of $u$ with respect to the variable $t$, and 
$u=u(x,t)$ is a real-valued unknown function. 
The coefficient of the damping term 
$a(x)$ satisfies 
$a\in C^2(\overline{\Omega})$, 
$a(x)>0$ on $\overline{\Omega}$ and 
\begin{align}\label{a0}
\lim_{|x|\to \infty}
\Big(\lr{x}^{\alpha}a(x)\Big)=a_0
\end{align}
with some constants $\alpha\in [0,1)$ and $a_0\in (0,\infty)$, 
where $\langle y\rangle=(1+|y|^2)^\frac{1}{2}$ for $y\in \R^N$. 
In this moment, the initial data $(u_0,u_1)$
are assumed to have compact supports in $\Omega$
and to satisfy the compatibility condition of order $k\geq 1$:
\begin{align}
\label{compat}
 (u_{\ell-1},u_{\ell})\in (H^2\cap H^1_0(\Omega))\times H^1_0(\Omega), 
 \quad
 \text{for all}\ \ell = 1,\ldots,k
\end{align}
where $u_{\ell}$ is successively defined by 
$u_{\ell}=\Delta u_{\ell-2}-a(x)u_{\ell-1}$ ($\ell=2,\ldots,k$). 
We note that existence and uniqueness of solution to the problem \eqref{dw}
have %%%%%
been discussed 
(see e.g., Ikawa \cite[Theorem 2]{Ikbook}). 

It is proved in Matsumura \cite{Ma76} that if $\Omega=\R^N$ and $a(x)\equiv 1$, 
then the solution $u$ of \eqref{dw} satisfies 
the energy decay estimate
\[
\int_{\R^N}(|\nabla u(x,t)|^2+|u_t(x,t)|^2)\,dx
\leq 
C(1+t)^{-\frac{N}{2}-1}\|(u_0,u_1)\|_{H^1\times L^2}^2,
\]
where the constant $C$ depends on 
the size of the supprot of initial data. 
Moreover, it is shown in Nishihara \cite{Nishihara03} that
$u$ has the same asymptotic behavior 
as the one of the problem
\begin{align}
\label{heat0}
	\left\{\begin{array}{ll}
	   v_t-\Delta v= 0,& x\in \R^N,\ t>0,
\\
	v(x,0) = u_0(x)+u_1(x),&x\in \R^N.
	\end{array}\right.
\end{align}
In particular, we have
\[
\|u(\cdot,t)-v(\cdot,t)\|_{L^2}=o(t^{-\frac{N}{4}})
\]
as $t\to\infty$. 
Energy decay 
properties 
of solutions to \eqref{dw} 
for general cases with 
$a(x)\geq \lr{x}^{-\alpha}$
($0\leq \alpha\leq 1$) 
have been dealt with 
by Matsumura \cite{Ma77}. 
On the other hand, Mochizuki \cite{Mo76} 
proved that 
if $0\leq a(x)\leq C\lr{x}^{-\alpha}$ for some 
$\alpha>1$, then the energy of the solution to \eqref{dw} 
does not vanish as $t\to \infty$ 
for suitable initial data. 
(The solution has an asymptotic behavior similar 
to the solution of the usual wave equation without damping). 
Therefore one can expect that 
diffusion phenomena 
occur only when $a(x)\geq C\lr{x}^{-\alpha}$ for $\alpha\leq 1$.

In this paper, we discuss precise decay rates 
of the weighted energy 
\[
\int_{\R^N}(|\nabla u(x,t)|^2+|u_t(x,t)|^2)\Phi(x,t)\,dx
\] 
with 
a special weight function
\[
\Phi(x,t)=\exp\left(\beta\,\frac{A(x)}{1+t}\right)
\]
(for some $A\in C^2(\R^N)$ and $\beta>0$) 
which is introduced by Todorova and Yordanov \cite{ToYo09} based on the ideas in \cite{ToYo01} 
and in \cite{Ik05IJPAM}. 
They proved weighted energy estimates
\begin{gather*}
\int_{\R^N}a(x)|u(x,t)|^2\Phi(x,t)\,dx
\leq 
C(1+t)^{-\frac{N-\alpha}{2-\alpha}+\ep}, %%%%%
\\
\int_{\R^N}(|\nabla u(x,t)|^2+|u_t(x,t)|^2)\Phi(x,t)\,dx
\leq 
C(1+t)^{-\frac{N-\alpha}{2-\alpha}-1+\ep}
\end{gather*}
when $a(x)$ is radially symmetric
and satisfies \eqref{a0}. 
After that, Radu, Todorova and Yordanov 
\cite{RTY09}
extended it to higher-order derivatives. 
In \cite{Wa14}, the second author proved diffusion phenomena for \eqref{dw} with $\Omega=\R^N$ and $a(x)=\lr{x}^{-\alpha}$ $(\alpha\in [0,1))$
by comparing the solution of 
the following problem
\begin{align}
\label{heat}
	\left\{\begin{array}{ll}
	   a(x)v_t-\Delta v= 0,& x\in \R^N,\ t>0,
\\
	v(x,0) = u_0(x)+\dfrac{1}{a(x)}u_1(x),&x\in \R^N.
	\end{array}\right.
\end{align}
In \cite{So_Wa1}, diffusion phenomena 
for \eqref{dw} with an exterior domain 
and for general radially symmetric damping term
are obtained. 
 However, the weighted energy estimates 
 and diffusion phenomena for \eqref{dw} 
 with {\bf non-radially symmetric damping} 
 are still remaining open. 
 The difficulty seems to come from 
 the choice of auxiliary function $A$ in the weighted energy, 
 which strongly depends on  
 the existence of positive solution
 to the Poisson equation $\Delta A(x)=a(x)$. 
%%%%%
In fact, an example of non-existence of 
 positive solution to
$\Delta A = a$ for non-radial $a(x)$ is shown in \cite{So_Wa1}.
Radu, Todorova and Yordanov \cite{RTY10} 
considered the case $\Omega = \R^N$
and used a solution $A_*(x)$
of $\Delta A_* = a_1 (1+|x|)^{-\alpha}$ 
with $a_1>0$ satisfying
$a_1 (1+|x|)^{-\alpha} \ge a(x)$ for $x \in \R^N$,
that is, $A_*(x)$ is a subsolution of 
the equation $\Delta A = a$.
In general one cannot obtain the optimal decay estimate
via this choice because of the luck of 
the precise behavior of $a(x)$ at the spatial infinity
which can be expected to determine 
the precise decay late of weighted energy estimates. 
Our main idea to overcome this difficulty is
to weaken the equality $\Delta A = a$ and consider the inequality
$(1-\varepsilon) a \le \Delta A \le (1+\varepsilon) a$,
and to construct a solution having appropriate behavior,
we employ a cut-off argument.
 
 The aim of this paper is to give a proof of 
 Ikehata--Todorova--Yordanov type weighted energy estimates  
 for \eqref{dw} 
 with non-radially symmetric damping 
 and to obtain diffusion phenomena for \eqref{dw} 
 under the compatibility condition of order $1$ and 
 the condition \eqref{a0} (without any restriction). 
 
%\section{result}

 This paper is originated as follows. 
 In Section 2, we discuss related elliptic 
 and parabolic problems. 
 The weighted energy estimates for \eqref{dw}  
 are established in Section 3 (Proposition \ref{main}). 
 Section 4 is devoted to show diffusion phenomena (Proposition \ref{dp}).

\section{Related elliptic and parabolic problems}
\subsection{An elliptic problem for weighted energy estimates}

%Since the non-existence of positive solutions to 
%$\Delta A(x)=a(x)$ is mentioned in \cite{So_Wa1}, 
%we modify the above equation elliptic problem as 
As we mentioned above, in general,
existence 
of positive solutions to the Poisson equation 
$\Delta A(x) = a(x)$
is false for non-radial $a(x)$.
Thus, we weaken this equation and consider the following inequality
\begin{align}
\label{ell.aux}
(1-\ep)a(x)\leq \Delta A(x)\leq (1+\ep)a(x), \quad x\in \Omega,
\end{align}
where $\ep\in (0,1)$ is a parameter. 
Here we construct a positive solution $A$ of 
\eqref{ell.aux}
satisfying 
\begin{gather}
\label{A1}
A_{1\ep}\lr{x}^{2-\alpha}
\leq A(x)
\leq 
A_{2\ep}\lr{x}^{2-\alpha},
\\
\label{A2}
\frac{|\nabla A(x)|^2}{a(x)A(x)}\leq \frac{2-\alpha}{N-\alpha}+\ep
\end{gather}
for some constants $A_{1\ep}, A_{2\ep}>0$.
\begin{lemma}\label{sol.ell}
For every $\ep\in (0,1)$, there exists 
$A_\ep\in C^2(\overline{\Omega})$
such that $A_{\ep}$ satisfies \eqref{ell.aux}--\eqref{A2}.
\end{lemma}

\begin{proof}
Firstly, we extend $a(x)$ 
as a positive function in $C^2(\R^N)$; note that 
this is possible by virtue of the smoothness of $\pa\Omega$. 
To simplify the notation, 
we use the same symbol $a(x)$ as a function defined on $\R^N$.
We construct a solution of approximated equation 
\[
\Delta A_{\ep}(x) = a_\ep(x), \quad x\in \R^N
\]
for some $a_{\ep}\in C^2(\R^N)$ satisfying 
\begin{equation}\label{a_ep}
(1-\ep)a(x)\leq a_{\ep}(x)\leq (1+\ep) a(x), 
\quad x\in \R^N.
\end{equation}
Noting \eqref{a0}, we divide $a(x)$ as
$a(x)=b_1(x)+b_2(x)$ with 
\begin{align*}
b_1(x)&=\Delta 
\left(\frac{a_0}{(N-\alpha)(2-\alpha)}\lr{x}^{2-\alpha}\right)
=a_0\lr{x}^{-\alpha}+\frac{a_0\alpha }{N-\alpha}\lr{x}^{-\alpha-2}, 
\\
b_2(x)&=a(x)-a_0\lr{x}^{-\alpha}-\frac{a_0\alpha}{N-\alpha}\lr{x}^{-\alpha-2}.
\end{align*}
Then we have 
\begin{align}\label{error}
\lim_{|x|\to \infty}\left(\frac{b_2(x)}{a(x)}\right)
=
\lim_{|x|\to \infty}
\left[
 \frac{1}{\lr{x}^{\alpha}a(x)}
 \left(\lr{x}^{\alpha}a(x)-a_0-\frac{a_0\alpha}{N-\alpha}\lr{x}^{-2}
 \right)
 \right]
=0.
\end{align}

Let $\ep\in (0,1)$ be fixed. Then by \eqref{error} 
there exists  a constant $R_\ep>0$ such that 
$|b_2(x)|\leq \ep a(x)$ 
for $x\in \R^N\setminus B(0,R_\ep)$. 
Here we introduce a cut-off function 
$\eta_\ep\in C_c^\infty(\R^N,[0,1])$ such that $\eta_\ep\equiv 1$ on $B(0,R_\ep)$. Define 
\[
a_{\ep}(x):=b_1(x)+\eta_\ep(x)b_2(x)
=a(x)-(1-\eta_\ep(x))b_2(x), \quad x\in \R^N.
\]
Then $a_\ep(x)=a(x)$ on $B(0,R_\ep)$ and for $x\in \R^N\setminus B(0,R_\ep)$, 
\[
\left|
\frac{a_\ep(x)}{a(x)}-1
\right|
=(1-\eta_\ep(x))\frac{|b_2(x)|}{a(x)}\leq \ep
\]
and therefore \eqref{a_ep} is verified. 

Next we define  
\begin{align*}
B_{1\ep}(x)
&:=\frac{a_0}{(N-\alpha)(2-\alpha)}\lr{x}^{2-\alpha}, 
\quad x\in \R^N,
\\
B_{2\ep}(x)
&:=-
\int_{\R^N}\mathcal{N}(x-y)\eta_\ep(y)b_2(y)\,dy, 
\quad x\in \R^N,
\end{align*}
where $\mathcal{N}$ is the Newton potential given by 
\[
\mathcal{N}(x)
=
\begin{cases}
\dfrac{1}{2\pi}\log\dfrac{1}{|x|} & \text{if}\ N=2,
\\[10pt]
\dfrac{\Gamma(\frac{N}{2}+1)}{N(N-2)\pi^{\frac{N}{2}}}
|x|^{2-N} & \text{if}\ N\geq 3.
\end{cases}
\]
Then we easily see that 
$\Delta B_{1\ep}(x)=b_1(x)$ and $\Delta B_{2\ep}=\eta_\ep(x) b_2(x)$.
Moreover, noting that 
${\rm supp}\,(\eta_\ep b_2)$ is compact, 
we see from a direct calculation that 
there exist a constant $M_\ep>0$ such that 
\[
\left|
B_{2\ep}(x)
\right|
\leq 
\begin{cases}
M_\ep(1+\log \lr{x})&\text{if}\ N=2,
\\
M_\ep\lr{x}^{2-N}&\text{if}\ N\geq 3,
\end{cases}
\qquad 
\left|
\nabla B_{2\ep}(x)
\right|
\leq M_\ep\lr{x}^{1-N},
\quad x\in \R^N.
\]
This yields that 
$B_\ep:=B_{1\ep}+B_{2\ep}$ is bounded from below
and positive for $x\in \R^N$ 
with sufficiently large $|x|$. 
Moreover, we have
\[
\lim_{|x|\to \infty}
\Big(\lr{x}^{\alpha-2}B_{\ep}(x)\Big)=\frac{a_0}{(N-\alpha)(2-\alpha)}
\]
and 
\begin{align*}
&\lim_{|x|\to \infty}
\left(
\frac{|\nabla B_\ep(x)|^2}{a(x)B_\ep(x)}
\right) \\
&\quad =
\lim_{|x|\to \infty}
\left(
\frac{1}{\lr{x}^{\alpha}a(x)}\cdot\frac{1}{\lr{x}^{\alpha-2}B_\ep(x)}
\left|\frac{a_0}{N-\alpha}\lr{x}^{-1}x+\lr{x}^{\alpha-1}\nabla B_{2\ep}(x)\right|^2
\right)
\\
&\quad =
\frac{2-\alpha}{N-\alpha}.
\end{align*}
Using the same argument as in 
the proof of \cite[Lemma 3.1]{So_Wa1}, 
we can see that there exists a constant $\lambda_\ep\geq 0$ 
such that  
$A_{\ep}(x):=\lambda_\ep+B_{\ep}(x)$ 
satisfies \eqref{ell.aux}-\eqref{A2}.  
\end{proof}

\subsection{A parabolic problem for diffusion phenomena}

Here we consider $L^p$-$L^q$ type estimates 
for solutions to 
the initial-boundary value problem 
of the following parabolic equation 
\begin{align}
\label{heat_f}
	\left\{\begin{array}{ll}
	a(x)w_{t}-\Delta w = 0,& x\in \Omega,\ t>0,
\\
	w(x,t)=0,&x\in \partial \Omega, \ t>0,
\\
	w(x,0) = f(x),&x\in \Omega.
	\end{array}\right.
\end{align} 
Here we introduce a weighted $L^p$-spaces 
\begin{align*}
   L^p_{d\mu}
:=
   \left\{
      f\in L^p_{\rm loc}(\Omega)
   \;;\;
      \|f\|_{L^p_{d\mu}}:=
      \left(\int_{\Omega}
         |f(x)|^p a(x)
      \,dx\right)^{\frac{1}{p}}<\infty
   \right\},	 
\quad 
   1\leq p<\infty
\end{align*}
which is quite reasonable because 
the corresponding elliptic operator $a(x)^{-1}\Delta$
can be regarded as a symmetric operator in 
$L^2_{d\mu}$.

The $L^p$-$L^q$ type estimates for 
the semigroup associated with  
the Friedrichs' extension $-L_*$ (in $L^2_{d\mu}$)
of $-a(x)^{-1}\Delta$ are stated in {\cite{So_Wa1}}. 
The proof is based on Beurling--Deny's criterion
and Gagliardo--Nirenberg inequality. 
%%%%%%%%%%%%%%%%%%%%%%%%%%%%%%%%%%%%%%%%%%%%%%%%%%%%%%%%%%%%
\begin{proposition}[{\cite[Proposition 2.6]{So_Wa1}}]
%%%%%%%%%%%%%%%%%%%%%%%%%%%%%%%%%%%%%%%%%
\label{embedding2}
Let $e^{tL_*}$ be a semigroup generated by $L_*$. 
For every $f\in L^1_{d\mu}\cap L^2_{d\mu}$, we have
\begin{equation}
\label{1-2}
	\|e^{tL_*}f\|_{L^2_{d\mu}}\leq Ct^{-\frac{N-\alpha}{2(2-\alpha)}}\|f\|_{L^1_{d\mu}}
\end{equation}
and
\begin{align}
\label{1-3}
	\|L_{*} e^{tL_{\ast}}f\|_{L^2_{d\mu}}
		\leq Ct^{-\frac{N-\alpha}{2(2-\alpha)}-1}\|f\|_{L^1_{d\mu}}.
\end{align}
\end{proposition}%%%%%%%%%%%%%%%%%%%%%%%%%%%%%%%%%%%%%%%%%%%
%%%%%%%%%%%%%%%%%%%%%%%%%%%%%%%%%%%%%%%%%%%%%%%%%%%%%%%%%%%%

\section{Weighted energy estimates}

In this section we establish weighted energy estimates 
for solutions of \eqref{dw} by introducing 
Ikehata--Todorova--Yordanov type weight function 
with an auxiliary function $A_\ep$ constructed in Subsection 2.1.

To begin with, let us recall 
the finite speed propagation property of the wave equation
(see \cite{Ikbook}).
%%%%%%%%%%%%%%%%%%%%%%%%%%%%%%%%%%%%%
\begin{lemma}[Finite speed of propagation]
\label{fp}
Let $u$ be the solution of \eqref{dw} with the initial data
$(u_0,u_1)$
satisfying
${\rm supp}\,(u_0,u_1) \subset \overline{B}(0,R_0)=
\{ x \in \Omega ; |x| \le R_0 \}$.
Then, one has
\[
	{\rm supp}\,u(\cdot,t)\subset \{x\in \Omega\;;\;|x|\leq R_0+t\}
\]
and therefore
$|x|/(R_0+1+t) \leq 1$ for $t \ge 0$
and $x\in {\rm supp}\,u(\cdot, t)$.
\end{lemma}
%%%%%%%%%%%%%%%%%%%%%%%%%%%%%%%%%%%%%

Before introducing a weight function, 
we also recall two identities 
for partial energy functionals 
proved in \cite{So_Wa1}.
%%%%%%%%%%%%%%%%%%%%%%%%%%%%%%%%%%%%%%%%%%%%%%%%
\begin{lemma}[{\cite[Lemma 3.7]{So_Wa1}}]
%%%%%%%%%%%%%%%%%%%%%%%%%%%%%%%%%%%%%%%%
\label{e0-0}
Let $\Phi\in C^2(\overline{\Omega}\times [0,\infty))$ 
satisfy $\Phi>0$ and $\pa_t\Phi<0$
and let $u$ be a solution of \eqref{dw}. Then
\begin{align*}
	\frac{d}{dt}\left[ \int_{\Omega} \Big(|\nabla u|^2+|u_t|^2\Big) \Phi\,dx \right]
	&= \int_{\Omega} (\pa_t\Phi)^{-1} \big| \pa_t\Phi\nabla u -u_t\nabla\Phi \big|^2 \,dx\\
	&\  + \int_{\Omega}
		\Big( -2a(x)\Phi+\pa_t\Phi - (\pa_t\Phi)^{-1}|\nabla\Phi|^2 \Big)|u_t|^2 \,dx.   
\end{align*}
\end{lemma}%%%%%%%%%%%%%%%%%%%%%%%%%%%%%%%%%%%%%%%%%%%%%%%%

\begin{lemma}[{\cite[Lemma 3.9]{So_Wa1}}]
%%%%%%%%%%%%%%%%%%%%%%%%%%%%%%%%%%%%%%%%%%%%%%
\label{e1-0}
Let $\Phi\in C^2(\overline{\Omega}\times [0,\infty))$ 
satisfy $\Phi>0$ and $\pa_t\Phi<0$
and let $u$ be a solution to \eqref{dw}.
Then, we have
\begin{align*}
	\frac{d}{dt} \left[ \int_{\Omega} \Big(2uu_t+a(x)|u|^2\Big) \Phi\,dx  \right]
	&= 2\int_{\Omega} uu_{t} (\pa_t\Phi)\,dx
	+ 2\int_{\Omega} |u_t|^2 \Phi\,dx  
	- 2\int_{\Omega} |\nabla u|^2 \Phi\,dx
	\\
	&\quad + \int_{\Omega} \big(a(x)\pa_t\Phi+\Delta\Phi\big)|u|^2\,dx.
\end{align*}
\end{lemma}%%%%%%%%%%%%%%%%%%%%%%%%%%%%%%%%%%%%%%%%%%%%%%%

Here we introduce 
a weight function for weighted energy estimates, 
which is a modification of the one 
in Todorova-Yordanov \cite{ToYo09}. 
%%%%%%%%%%%%%%%%%%%%%%%%%%%%%%%%%%%%%%%%%%%%%%%%%%%%%%%%%%%
\begin{definition}
Define $h:=\frac{2-\alpha}{N-\alpha}$ 
and for $\ep\in (0,1)$,  
\begin{equation}
\Phi_{\ep}(x,t)
=
\exp\left(
\frac{1}{h+2\ep}\,\frac{A_\ep(x)}{1+t}
\right),
\end{equation}
where $A_{\ep}$ is given in Lemma \ref{sol.ell}. 
And define for $t\geq 0$,
\begin{gather}
\label{en_xta}
E_{\pa x}(t;u)
:=
\int_{\Omega}|\nabla u|^2\Phi_\ep\,dx,
\quad
E_{\pa t}(t;u)
:=
\int_{\Omega}|u_t|^2\Phi_\ep\,dx,
\\
\label{en_*A}
E_{a}(t;u)
:=
\int_{\Omega}a(x)|u|^2\Phi_\ep\,dx,
\quad
E_{*}(t;u)
:=
2\int_{\Omega}uu_t\Phi_\ep\,dx,
\end{gather}
and also define
$E_1(t;u):=E_{\pa x}(t;u)+E_{\pa t}(t;u)$ 
and $E_2(t;u):=E_{*}(t;u)+E_{a}(t;u)$.
\end{definition}

%{\color{red}
%\begin{remark}
%By finite propagation property 
%\[
%E_{a}(t;\pa_t u)
%=\int_{\Omega}a(x)|u_t|^2\Phi_\ep\,dx
%\leq 
%F(t;u)\leq 
%\left(1+\frac{A_{2\ep}(R_0+1)^2}{h+2\ep}\right)
%E_{a}(t;\pa_t u)
%\]
%\end{remark}}
%%%%%%%%%%%%%%%%%%%%%%%%%%%%%%%%%%%%%%%

Now we are in a position to 
state our main result for 
weighted energy estimates for solutions 
of \eqref{dw}.

\begin{proposition}\label{main}

Assume that $(u_0,u_1)$ satisfies 
${\rm supp}\,(u_0,u_1)\subset \overline{B}(0,R_0)$ 
and the compatibility condition of 
order $k_0\geq 1$. 
Let $u$ be a solution of the problem \eqref{dw}. 
For every $\delta>0$ and $0 \leq k\leq k_0-1$, 
there exist $\ep>0$ and $M_{\delta,k,R_0}>0$ such that 
%$\lambda (\ep)=\frac{1-5\ep}{(1+\ep)(h+2\ep)}$ 
%$(=h^{-1}-\delta_\ep)$
for every $t\geq 0$, 
\begin{align*}
&(1+t)^{\frac{N-\alpha}{2-\alpha}+2k+1-\delta}
\Big(E_{\pa x}(t;\pa_t^{k}u)+E_{\pa t}(t;\pa_t^{k}u)\Big)
+
(1+t)^{\frac{N-\alpha}{2-\alpha}+2k-\delta}E_a(t;\pa_t^{k}u) \\
&\quad \leq M_{\delta,k,R_0}\|(u_0,u_1)\|_{H^{k+1}\times H^k(\Omega)}^2.
\end{align*}
\end{proposition}

To prove, this, we prepare the following two lemmas.

\begin{lemma}%%%%%%%%%%%%%%%%%%%%%%%%%%%%%%%%%%%%%%%%%%%%%%
\label{lem_ha}
For $t\geq 0$, we have
\begin{align}
\label{hardy}
\frac{1-\ep}{h+2\ep}\,\frac{1}{1+t}
E_{a}(t;u)
	\leq  E_{\pa x}(t;u).
\end{align}
\end{lemma}
%%%%%%%%%%%%%%%%%%%%%%%%%%%%%%%%%%%%%%%%

\begin{proof}
As in the proof of \cite[Lemma 3.6]{So_Wa1}, 
by integration by parts we have
\[
\int_{\Omega}
\Delta(\log\Phi_\ep)|u|^2\Phi_{\ep}\,dx
=
\int_{\Omega}\left(\Delta\Phi_\ep-\frac{|\nabla\Phi_\ep|^2}{\Phi_\ep}\right)|u|^2\,dx
\leq 
\int_{\Omega}|\nabla u|^2\,\Phi_\ep\,dx.
\]
Noting that 
\[
\Delta(\log\Phi_\ep(x))=\frac{1}{h+2\ep}\,\frac{\Delta A_{\ep}(x)}{1+t}
\geq 
\frac{1-\ep }{h+2\ep}\,\frac{a(x)}{1+t},
\]
we have \eqref{hardy}.
\end{proof}

In order to clarify the effect of the finite propagation property, 
we now put
\[
a_1:=\inf_{x\in \Omega}\Big(\lr{x}^\alpha a(x)\Big).
\]
Then
\begin{lemma}
For $t\geq 0$, we have
\begin{align}
\label{E12-F}
E_{\pa t}(t;u)
&\leq \frac{1}{a_1}(R_0+1+t)^\alpha E_{a}(t;\pa_t u),
\\
\label{A/a}
\int_{\Omega} \frac{A_\ep(x)}{a(x)}|u_{t}|^2 \Phi_\ep \,dx
&\leq \frac{A_{2\ep}}{a_1}(R_0+1+t)^{2}
E_{\pa t}(t;u),
\\
\label{E21-F}
|E_{*}(t;u)|
&\leq 
\frac{2}{\sqrt{a_1}}(R_0+1+t)^{\frac{\alpha}{2}}
\sqrt{E_a(t;u)E_{\pa t}(t;u)}. 
\end{align}
\end{lemma}

\begin{proof}
By $a(x)^{-1}\leq a_1^{-1}\lr{x}^{\alpha}
\leq a_1^{-1}(1+|x|)^{\alpha}$ 
and the finite propagation property we have 
\begin{align*}
\int_{\Omega}|u_t|^2\Phi_\ep\,dx
=\int_{\Omega}\frac{a(x)}{a(x)}|u_t|^2\Phi_\ep\,dx
\leq \frac{1}{a_1}(R_0+1+t)^{\alpha}E_a(t;\pa_t u).
\end{align*}
Using the Cauchy-Schwarz inequality and the above inequality yields 
\eqref{A/a}:
\begin{align*}
\left|\int_{\Omega}uu_t\Phi_\ep\,dx\right|^2
&\leq 
\left(\int_{\Omega}|u|^2\Phi_\ep\,dx\right)
\left(\int_{\Omega}|u_t|^2\Phi_\ep\,dx\right)
\\
&
\leq 
\frac{(R_0+1+t)^{\alpha}}{a_1}\left(\int_{\Omega}a(x)|u|^2\Phi_\ep\,dx\right)
E_{\pa t}(t;u)
\\
&\leq 
\frac{(R_0+1+t)^{\alpha}}{a_1}
E_a(t;u)E_{\pa t}(t;u). 
\end{align*}
We can prove \eqref{E21-F} in a similar way. 
\end{proof}

\begin{lemma}
{\bf (i)}\ For every $t\geq 0$, we have
\begin{equation}\label{e0-1}
%	\frac{d}{dt}\left[ \int_{\Omega} \Big(|\nabla u|^2+|u_t|^2\Big) \Phi_\ep\,dx \right]
    \frac{d}{dt}E_1(t;u)
	\leq 
    -E_{a}(t;\pa_t u).
    %-F_{a}(t;\pa_t u).
%	-\int_{\Omega}
%	\left(a(x)+\frac{A_{\ep}(x)}{(h+2\ep)(1+t)^2}\right)
%	|u_t|^2\Phi_\ep
%	\,dx.
\end{equation}
{\bf (ii)}\ For every $\ep\in (0,\frac{1}{3})$ and $t\geq 0$, 
\begin{align}
%	\frac{d}{dt} \left[ \int_{\Omega} \Big(2uu_t+a(x)|u|^2\Big) \Phi_\ep\,dx  \right]
    \frac{d}{dt}E_{2}(t;u)
	&\leq 
	-\frac{1-3\ep}{1-\ep}E_{\pa x}(t;u)
\label{e1-1}
	+
	\left(\frac{2}{a_1}+\frac{A_{2\ep}(R_0+1)^2}{\ep a_1^2}\right)(R_0+1+t)^{\alpha}E_{a}(t;\pa_t u).
\end{align}
\end{lemma}

\begin{proof}
Noting \eqref{A2}, we have
\begin{align*}
&-2a(x)\Phi_\ep+\pa_t\Phi_\ep - (\pa_t\Phi_\ep)^{-1}
|\nabla \Phi_\ep|^2\\
&\quad=
\left(-2a(x)-\frac{A_{\ep}(x)}{(h+2\ep)(1+t)^2}
+\frac{1}{h+2\ep}\frac{|\nabla A_\ep(x)|^2}{A_\ep(x)}\right)\Phi_\ep
\\
&\quad\leq 
\left(
-2a(x)
%-\frac{A_{\ep}(x)}{(h+2\ep)(1+t)^2}
+\frac{h+\ep}{h+2\ep}a(x)\right)\Phi_\ep
\\
&\quad\leq 
-
a(x)\Phi_\ep.
\end{align*}
This implies \eqref{e0-1}. On the other hand, 
from \eqref{A2} and \eqref{ell.aux} we see
\begin{align*}
a(x)\pa_t\Phi_\ep+\Delta\Phi_\ep
&=
\frac{1}{h+2\ep}
\left(-\frac{a(x)A_{\ep}(x)}{(1+t)^2}
+\frac{|\nabla A_\ep(x)|^2}{(h+2\ep)(1+t)^2}
+\frac{\Delta A_{\ep}(x)}{1+t}
\right)\Phi_\ep
\\
&\leq 
\frac{1}{h+2\ep}
\left(-\frac{a(x)A_{\ep}(x)}{(1+t)^2}
+\frac{(h+\ep)a(x)A_{\ep}(x)}{(h+2\ep)(1+t)^2}
+\frac{(1+\ep)a(x)}{1+t}
\right)\Phi_\ep
\\
&
\leq
\left(
-\frac{\ep}{(h+2\ep)^2}\,\frac{a(x)A_{\ep}(x)}{(1+t)^2}
+\frac{1+\ep}{h+2\ep}\,\frac{a(x)}{1+t}
\right)
\Phi_\ep. 
\end{align*}
Therefore combining it with Lemma \ref{lem_ha}, 
we have 
\begin{align*}
&\int_{\Omega} \big(a(x)\pa_t\Phi_\ep+\Delta\Phi_\ep\big)|u|^2\,dx\\
&\quad \leq
\frac{1+\ep}{1-\ep}\int_{\Omega} |\nabla u|^2\Phi_\ep\,dx
-\frac{\ep}{(h+2\ep)^2}\,
\frac{1}{(1+t)^2}
\int_{\Omega} a(x)A_\ep(x)|u|^2\Phi_\ep\,dx.
\end{align*}
Using \eqref{A/a}, we have 
\begin{align*}
&2\int_{\Omega} uu_{t} (\pa_t\Phi_\ep)\,dx \\
&\quad =
-\frac{2}{h+2\ep}\frac{1}{(1+t)^2}\int_{\Omega} uu_{t} A_\ep(x)\Phi_\ep\,dx
\\
&\quad \leq
\frac{2}{h+2\ep}\frac{1}{(1+t)^2}
\left(\int_{\Omega} a(x)A_\ep(x)|u|^2\Phi_\ep\,dx\right)^\frac{1}{2}
\left(\int_{\Omega} \frac{A_\ep(x)}{a(x)}|u_{t}|^2 \Phi_\ep \,dx\right)^\frac{1}{2}
\\
&\quad \leq
\frac{2(R_0+1)}{h+2\ep}
\frac{1}{1+t}
\left(\int_{\Omega} a(x)A_\ep(x)|u|^2\Phi_\ep\,dx\right)^\frac{1}{2}
\left(
\frac{A_{2\ep}}{a_1}
E_{\pa t}(t;u)
\right)^\frac{1}{2}
\\
&\quad \leq
\frac{\ep}{(h+2\ep)^2}\,\frac{1}{(1+t)^2}
\int_{\Omega} a(x)A_\ep(x)|u|^2\Phi_\ep\,dx
+
\frac{A_{2\ep}(R_0+1)^2}{\ep a_1}
E_{\pa t}(t;u).
\end{align*}
Applying \eqref{E12-F}, we obtain \eqref{e1-1}.
\end{proof}

\begin{lemma}\label{est}
The following assertions hold:

\smallskip

\noindent{\bf (i)} 
Set 
$t_*(R_0,\alpha,m):=
\max\left\{\left(\frac{2m}{a_1}\right)^{\frac{1}{1-\alpha}},R_0+1\right\}$. 
Then for every $t,m\geq 0$ and $t_1\geq t_*(R_0,\alpha,m)$,  
\begin{align}
\label{e1-m}
\frac{d}{dt}
\Big((t_1+t)^{m}E_1(t;u)\Big)
\leq 
m
(t_1+t)^{m-1}E_{\pa x}(t;u)
- \frac{1}{2}(t_1+t)^{m}E_a(t;\pa_t u).
\end{align}
\noindent{\bf (ii)} 
for every $t,\lambda\geq 0$ and $t_2\geq R_0+1$,  
\begin{align}
\nonumber
&\frac{d}{dt}\Big((t_2+t)^{\lambda}E_2(t;u)\Big) \\
\nonumber
&\quad\leq 
\lambda(1+\ep)(t_2+t)^{\lambda-1}
E_a(t;u)
-\frac{1-3\ep}{1-\ep}
(t_2+t)^{\lambda}E_{\pa x}(t;u)
\\
\label{e2-l}
&\qquad
+
\left(
\frac{2}{a_1}+\frac{A_{2\ep}(R_0+1)^2}{\ep a_1^2}+
\frac{\lambda}{2\ep a_1^2 t_2^{1-\alpha}}
\right)(t_2+t)^{\lambda+\alpha}E_a(t;\pa_t u).
\end{align}
\noindent{\bf (iii)} 
In particular, setting 
\begin{align*}
\nu
&:=
\frac{4}{a_1}+\frac{2A_{2\ep}(R_0+1)^2}{\ep a_1^2}+
\frac{1}{4\ep a_1}, 
\\
t_{**}(\ep,R_0,\alpha,\lambda)
&:=
\max\left\{ \left(\frac{(1-\ep)(\lambda+\alpha)\nu}{\ep}\right)^{\frac{1}{1-\alpha}},\left(\frac{2(\lambda+\alpha)}{a_1}\right)^{\frac{1}{1-\alpha}},R_0+1\right\},
\end{align*}
one has that for $t,\lambda\geq 0$ and 
$t_3\geq t_{**}(\ep,R_0,\alpha,\lambda)$,  
\begin{align}
\nonumber
&\frac{d}{dt}
\Big(
\nu (t_3+t)^{\lambda+\alpha}E_1(t;u)
+(t_3+t)^{\lambda}E_2(t;u)
\Big)
\\
\label{e3-l}
&\quad \leq 
-\frac{1-4\ep}{1-\ep}
(t_3+t)^{\lambda}E_{\pa x}(t;u)
+
\lambda(1+\ep)(t_3+t)^{\lambda-1}
E_a(t;u).
\end{align}
\end{lemma}

\begin{proof}
{\bf (i)} 
Let $m\geq 0$ be fixed and let $t_1\geq t_{*}(R_0,\alpha,m)$. 
Using \eqref{e0-1} and \eqref{E12-F}, we have 
\begin{align*}
&(t_1+t)^{-m}
\frac{d}{dt}
\Big((t_1+t)^{m}E_1(t;u)\Big) \\
&\quad \leq 
\frac{m}{t_1+t}
E_{\pa x}(t;u) 
+
\frac{m}{t_1+t}
E_{\pa t}(t;u)
+\frac{d}{dt}E_1(t;u)
\\
&\quad \leq 
\frac{m}{t_1+t}
E_{\pa x}(t;u) 
+
\frac{m}{t_1+t}
E_{\pa t}(t;u)-E_a(t;\pa_t u)
\\
&\quad \leq 
\frac{m}{t_1+t}
E_{\pa x}(t;u) 
+ 
\left(\frac{m(R_0+1+t)^{\alpha}}{a_1(t_1+t)}
 - 1\right)E_a(t;\pa_t u).
\end{align*}
Therefore 
we obtain 
\eqref{e1-m}.

\noindent{\bf (ii)}  For $t\geq 0$, 
and $t\geq R_0+1$, 
\begin{align*}
&(t_2+t)^{-\lambda}
\frac{d}{dt}\Big((t_2+t)^{\lambda}E_2(t;u)\Big)
\\
&\quad \leq 
\frac{\lambda}{t_2+t}E_*(t;u)
+ 
\frac{\lambda}{t_2+t}E_a(t;u)+\frac{d}{dt}E_2(t;u)
\\	
&\quad \leq 
\frac{\lambda}{t_2+t}E_*(t;u)
+ 
\frac{\lambda}{t_2+t}E_a(t;u)
-\frac{1-3\ep}{1-\ep}
E_{\pa x}(t;u) \\
&\qquad +\left(\frac{2}{a_1}+\frac{A_{2\ep}(R_0+1)^2}{\ep a_1^2}\right)
(R_0+1+t)^{\alpha}E_a(t;\pa_t u).
\end{align*}
Noting that by \eqref{E21-F} and \eqref{E12-F}, 
\begin{align*}
\frac{\lambda}{t_2+t}E_{*}(t;u)
&\leq 
\frac{2\lambda(R_0+1+t)^{\alpha}}{a_1(t_2+t)}
\sqrt{E_a(t;u)
E_a(t;\pa_t u)}
\\
&\leq 
\frac{\lambda\ep}{t_2+t}
E_a(t;u)
+
\frac{\lambda}{\ep a_1^2}
\frac{(R_0+1+t)^{2\alpha}}{t_2+t}
E_a(t;\pa_t u)
\\
&\leq 
\frac{\lambda\ep}{t_2+t}
E_a(t;u)
+
\frac{\lambda}{\ep a_1^2t_2^{1-\alpha}}(t_2+t)^{\alpha}
E_a(t;\pa_t u),
\end{align*}
we deduce \eqref{e2-l}.

\smallskip 

\noindent{\bf (iii)}  
Combining \eqref{e1-m} with $m=\lambda+\alpha$ 
and \eqref{e2-l}, we have
for $t_3\geq t_{**}(\ep,R_0,\alpha,\lambda)$ and $t\geq 0$, 
\begin{align*}
&\frac{d}{dt}
\Big(
\nu(t_3+t)^{\lambda+\alpha}E_1(t;u)
+(t_3+t)^{\lambda}E_2(t;u)
\Big)
\\
&\leq 
\left(
\nu(\lambda+\alpha)(t_3+t)^{\alpha-1}
-\frac{1-3\ep}{1-\ep}
\right)(t_3+t)^{\lambda}E_{\pa x}(t;u)
+\lambda(1+\ep)(t_3+t)^{\lambda-1}
E_a(t;u)
\\
&\quad
+\left(
\frac{2}{a_1}+\frac{A_{2\ep}(R_0+1)^2}{\ep a_1^2}+
\frac{\lambda}{2\ep a_1^2 t_3^{1-\alpha}}
- \frac{\nu}{2}\right)(t_3+t)^{\lambda+\alpha}E_a(t;\pa_t u)
\\
&\leq 
-
\frac{1-4\ep}{1-\ep}
(t_3+t)^{\lambda}E_{\pa x}(t;u)
+\lambda(1+\ep)(t_3+t)^{\lambda-1}
E_a(t;u).
\end{align*}
This proves the assertion.
\end{proof}

\begin{proof}[Proof of Proposition \ref{main}]
Firstly, by \eqref{E21-F} we observe that 
\begin{align*}
\nu (t_3+t)^{\alpha}E_1(t;u)
+E_2(t;u)
&\geq 
\frac{4}{a_1}(t_3+t)^{\alpha}E_1(t;u)
-
|E_*(t;u)|+E_a(t;u)
\\
&\geq 
\frac{4}{a_1}(t_3+t)^{\alpha}E_{\pa t}(t;u) \\
&\quad -
\frac{2}{\sqrt{a_1}}(t_3+t)^{\frac{\alpha}{2}}
\sqrt{E_a(t;u)E_{\pa t}(t;u)}
+E_a(t;u)
\\
&\geq \frac{3
}{4}E_a(t;u).
\end{align*}
By using the above estimate, 
we prove the assertion via mathematical induction. 

\noindent
{\bf Step 1 ($k=0$).}\ 
By \eqref{e3-l} 
using Lemma \ref{lem_ha} implies that 
\begin{align*}
&\frac{d}{dt}
\Big(
\nu (t_3+t)^{\lambda+\alpha}E_1(t;u)
+(t_3+t)^{\lambda}E_2(t;u)
\Big) \\
&\quad \leq 
\left(
-\frac{1-4\ep}{1-\ep}
+
\frac{\lambda(1+\ep)(h+2\ep)}{1-\ep}
\right)
(t_3+t)^{\lambda}E_{\pa x}(t;u).
\end{align*}
Therefore taking 
$\lambda_0=\frac{(1-\ep)(1-4\ep)}{(1+\ep)(h+2\ep)}$, 
($\lambda_0 \uparrow h^{-1}$ as $\ep\downarrow 0$)
we have
\begin{align*}
\frac{d}{dt}
\Big(
\nu (t_3+t)^{\lambda_0+\alpha}E_1(t;u)
+(t_3+t)^{\lambda_0}E_2(t;u)
\Big)
\leq 
-\frac{\ep (1-4\ep)}{1-\ep}
(t_3+t)^{\lambda_0}E_{\pa x}(t;u).
\end{align*}
Integrating over $(0,t)$ with respect to $t$, we see
\begin{align*}
&\frac{3}{4}(t_3+t)^{\lambda_0}E_a(t;u)
+
\frac{\ep(1-4\ep)}{1-\ep}
\int_{0}^t(t_3+s)^{\lambda_0}E_{\pa x}(s;u)\,ds
\\
&\leq
\nu (t_3+t)^{\lambda_0+\alpha}E_1(t;u)
+(t_3+t)^{\lambda_0}E_2(t;u)
+
\frac{\ep(1-4\ep)}{1-\ep}
\int_{0}^t(t_3+s)^{\lambda_0}E_{\pa x}(s;u)\,ds
\\
&
\leq 
\nu t_3^{\lambda_0+\alpha}E_1(0;u)
+t_3^{\lambda_0}E_2(0;u).
\end{align*}
Using \eqref{e1-m} with $m=\lambda_0+1$ 
and integrating over $(0,t)$, we obtain
\begin{align*}
&(t_3+t)^{\lambda_0+1}E_1(t;u)
+\frac{1}{2}
\int_{0}^t(t_3+s)^{\lambda_0+1}E_a(s;\pa_t u)\,ds
\\
&\quad \leq 
t_3^{\lambda_0+1}E_1(0;u)
+
(\lambda_0+1)
\int_0^t
(t_3+s)^{\lambda_0}E_{\pa x}(s;u)\,ds
\\
&\quad \leq
t_3^{\lambda_0+1}E_1(0;u)+\frac{(\lambda_0+1)(1-\ep)}
{\ep(1-4\ep)}
\Big(
\nu t_3^{\lambda_0+\alpha}E_1(0;u)
+t_3^{\lambda_0}E_2(0;u)
\Big).
\end{align*}
This proves the desired assertion with $k=0$ 
and also the integrability of $(t_3+s)^{\lambda_0+1}E_a(s;\pa_t u)$.

\smallskip 

\noindent{\bf Step 2 ($1<k\leq k_0-1$).}\ 
Suppose that 
for every $t\geq 0$, 
\[
(1+t)^{\lambda_0+2k-1}
E_{1}(t;\pa_t^{k-1}u)
+
(1+t)^{\lambda_0+2k-2}E_a(t;\pa_t^{k-1}u)
\leq M_{\ep,k-1}\|(u_0,u_1)\|_{H^{k}\times H^{k-1}(\Omega)}^2
\]
and additionally, 
\[
\int_{0}^t(1+s)^{\lambda_0+2k-1}E_a(s;\pa_t^{k}u)\,ds
\leq M_{\ep, k-1}'\|(u_0,u_1)\|_{H^{k}\times H^{k-1}(\Omega)}^2.
\]
Since the initial value $(u_0,u_1)$ 
satisfies the compatibility condition
of order $k$, 
$\pa_t^{k} u$ is also a solution of 
\eqref{dw} with replaced $(u_0,u_1)$ with $(u_{k-1},u_{k})$.
Applying \eqref{e3-l} with $\lambda=\lambda_0+2k$,  
putting $t_{3k}=t_{**}(\ep,R_0,\alpha,\lambda_0+2k)$ 
(see Lemma \ref{est}\ {\bf (iii)})
and integrating over $(0,t)$, we have
\begin{align*}
&
\frac{3}{4}
(t_{3k}+t)^{\lambda_0+2k}E_a(t;\pa_t^{k} u)
+\frac{1-4\ep}{1-\ep}
\int_0^t(t_{3k}+s)^{\lambda_0+2k}E_{\pa x}(s;\pa_t^{k} u)\,ds
\\
&\quad\leq 
\nu (t_{3k}+t)^{\lambda_0+2k+\alpha}E_1(t;\pa_t^{k} u)
+(t_{3k}+t)^{\lambda_0+2k}E_2(t;\pa_t^{k} u) \\
&\qquad +\frac{1-4\ep}{1-\ep}
\int_0^t(t_{3k}+s)^{\lambda_0+2k}E_{\pa x}(s;\pa_t^{k} u)\,ds
\\
&\quad \leq 
\nu t_{3k}^{\lambda_0+2k+\alpha}E_1(0;\pa_t^{k} u)
+t_{3k}^{\lambda_0+2k}E_2(0;\pa_t^{k} u) \\
&\qquad +
(\lambda_0+2k)(1+\ep)
\int_0^t(t_{3k}+s)^{\lambda_0+2k-1}
E_a(s;\pa_t^{k} u)\,ds
\\
&\quad \leq 
\nu t_{3k}^{\lambda_0+2k+\alpha}E_1(0;\pa_t^{k} u)
+t_{3k}^{\lambda_0+2k-1}E_2(0;\pa_t^{k} u) \\
&\qquad +
(\lambda_0+2k)(1+\ep)
M_{\ep, k-1}'\|(u_0,u_1)\|_{H^{k}\times H^{k-1}(\Omega)}^2.
\end{align*}
Moreover, from \eqref{e1-m} with $m=\lambda_0+2k+1$
we have 
\begin{align*}
&
(t_{3k}+t)^{\lambda_0+2k+1}E_1(t;\pa_t^{k}u)
+
\frac{1}{2}
\int_0^t(t_{3k}+s)^{\lambda_0+2k+1}E_a(s;\pa_t^{k+1}u)
\,ds
\\
&\quad \leq 
t_{3k}^{\lambda_0+2k+1}E_1(0;\pa_t^{k}u)
+
(\lambda_0+2k+1)
\int_0^t
(t_{3k}+s)^{\lambda_0+2k}E_{\pa x}(s;\pa_t^{k}u)
\,ds
\\
&\quad \leq M''_{\ep,k}
\Big(
E_1(0;\pa_t^{k} u)+E_2(0;\pa_t^{k} u)+\|(u_0,u_1)\|_{H^{k}\times H^{k-1}(\Omega)}^2
\Big)
\end{align*}
with some constant $M''_{\ep,k}>0$.
By induction we obtain the desired inequalities for all $k\leq k_0-1$. 
\end{proof}

\section{Diffusion phenomena as an application of weighted energy estimates}

\begin{proposition}\label{dp}
Assume that $(u_0,u_1)\in (H^2\cap H^1_0(\Omega))\times H^1_0(\Omega)$ 
and suppose that ${\rm supp}\,(u_0,u_1)\subset \overline{B}(0,R_0)$. 
Let $u$ be the solution of \eqref{dw}. Then for every 
$\ep>0$, there exists a constant $C_{\ep,R_0}>0$ such that 
\[
\Big\|u(\cdot,t) - e^{tL_{\ast}}[ u_0 + a(\cdot)^{-1}u_1]\Big\|_{L^2_{d\mu}}
\leq 
C_{\ep,R_0}(1+t)^{-\frac{N-\alpha}{2(2-\alpha)}-\frac{1-\alpha}{2-\alpha}+\ep}
\|(u_0,u_1)\|_{H^{2}\times H^1}.
\]
\end{proposition}

To prove Proposition \ref{dp} we use the following lemma stated in 
\cite[Section 4]{So_Wa1}. 
\begin{lemma}\label{Duhamel}
Assume that 
$
	(u_0,u_1) \in (H^2\cap H^1_0(\Omega))
		\times H^1_0(\Omega)
$
and
suppose that ${\rm supp}\,(u_0,u_1)\subset \{ x\in \Omega ; |x| \le R_0 \}$. 
Then for every $t\geq 0$, 
\begin{align}
\nonumber
	u(x,t) - e^{tL_{\ast}}[ u_0 + a(\cdot)^{-1}u_1]
	 &= - \int_{t/2}^t e^{(t-s)L_{\ast}}[ a(\cdot)^{-1}u_{tt}(\cdot, s) ] ds\\
\nonumber
	&\quad -e^{\frac{t}{2}L_{\ast}} [ a(\cdot)^{-1}u_{t}(\cdot, t/2) ] \\
\label{eq_du2}
	&\quad -\int_0^{t/2} L_{\ast}  e^{(t-s)L_{\ast}}[ a(\cdot)^{-1}u_{t}(\cdot, s) ] ds,
\end{align}
where
$L_{*}$
is the (negative) Friedrichs extension of
$-L=-a(x)^{-1}\Delta$ in $L^2_{d\mu}$.
\end{lemma}

\begin{proof}[Proof of Proposition \ref{dp}]
First we show the assertion for $(u_0,u_1)$ 
satisfying the compatibility condition of order $2$. 
Taking $L^2_{d\mu}$-norm of both side, we have
\[
\Big\|u(x,\cdot) - e^{tL_{\ast}}[ u_0 + a(\cdot)^{-1}u_1]\Big\|_{L^2_{d\mu}}
\leq 
\mathcal{J}_1(t)
+\mathcal{J}_2(t)
+\mathcal{J}_3(t),
\]
where
\begin{align*}
\mathcal{J}_1(t)&:=
\int_{t/2}^t 
\big\|e^{(t-s)L_{\ast}}[ a(\cdot)^{-1}u_{tt}(\cdot, s)]\big\|_{L^2_{d\mu}} ds, 
\\
\mathcal{J}_2(t)&:=\big\|e^{\frac{t}{2}L_{\ast}} [ a(\cdot)^{-1}u_{t}(\cdot, t/2)]\big\|_{L^2_{d\mu}},
\\
\mathcal{J}_3(t)&:=\int_0^{t/2} \big\|L_{\ast}  e^{(t-s)L_{\ast}}[ a(\cdot)^{-1}u_{t}(\cdot, s)]\big\|_{L^2_{d\mu}} ds.
\end{align*}
Noting that for $x\in \Omega$, 
\[
a(x)^{-1}\Phi_\ep(x,t)^{-1}\leq 
\frac{1}{a_1}\lr{x}^{\alpha}
\exp\left(-\frac{A_{1\ep}}{h+2\ep}\,\frac{\lr{x}^{2-\alpha}}{1+t}\right)
\leq 
\frac{1}{a_1}
\left(\frac{\alpha(h+2\ep)}{(2-\alpha)eA_{1\ep}}\right)^{\frac{\alpha}{2-\alpha}}
(1+t)^{\frac{\alpha}{2-\alpha}},
\]
we see that for $k=0,1$, 
\begin{align*}
\big\|a(\cdot)^{-1}\pa_t^{k+1}u(\cdot, s)\big\|_{L^2_{d\mu}}^2
&=
  \int_{\Omega}a(x)^{-1}|\pa_t^{k+1}u(\cdot, s)|^2\,dx
\\
&\leq 
  \|a(\cdot)^{-1}\Phi_\ep(\cdot,t)^{-1}\|_{L^\infty(\Omega)}
  \int_{\Omega}|\pa_t^{k+1}u(\cdot, s)|^2\Phi_\ep\,dx
\\
&\leq 
  \widetilde{C}(1+t)^{\frac{\alpha}{2-\alpha}}E_{\pa t}(t,\pa_t^{k}u)
\\
&\leq 
  \widetilde{C}M_{\ep,k}(1+t)^{-\lambda_0-\frac{2-2\alpha}{2-\alpha}-2k}
  \|(u_0,u_1)\|_{H^{k+1}\times H^k}^2.
\end{align*}
Therefore from Proposition \ref{main} with $k=1$ and $k=0$ we have
\begin{align*}
\mathcal{J}_1(t)
&\leq 
\int_{t/2}^t 
\big\|a(\cdot)^{-1}u_{tt}(\cdot, s)\big\|_{L^2_{d\mu}} ds
\\
&\leq 
  \sqrt{\widetilde{C}M_{1}}
  \|(u_0,u_1)\|_{H^{2}\times H^1}
  \int_{t/2}^t (1+s)^{-\frac{\lambda_0}{2}-\frac{1-\alpha}{2-\alpha}-1} ds
\\
&\leq 
  \frac{2(2-\alpha)}{\lambda_0(2-\alpha)+1-\alpha}
  \sqrt{\widetilde{C}M_{\ep,1}}
  (1+t)^{-\frac{\lambda_0}{2}-\frac{1-\alpha}{2-\alpha}}
  \|(u_0,u_1)\|_{H^{2}\times H^1}
\end{align*}
and 
\begin{align*}
\mathcal{J}_2(t)
&\leq \big\|a(\cdot)^{-1}u_{t}(\cdot, t/2)\big\|_{L^2_{d\mu}}
\leq 
  \sqrt{\widetilde{C}M_{\ep,0}}(1+t)^{-\frac{\lambda_0}{2}-\frac{1-\alpha}{2-\alpha}}
\|(u_0,u_1)\|_{H^{1}\times L^2}.
\end{align*}
Moreover, by Lemma \ref{embedding2}, 
we see by Cauchy--Schwarz inequality that for $t\geq 1$, 
\begin{align*}
\mathcal{J}_3(t)
&\leq 
C\int_0^{t/2} 
(t-s)^{-\frac{N-\alpha}{2(2-\alpha)}-1}
\big
\|a(\cdot)^{-1}u_{t}(\cdot, s)\big\|_{L^1_{d\mu}} ds
\\
&\leq 
C
\left(
\frac{t}{2}
\right)^{-\frac{N-\alpha}{2(2-\alpha)}-1}
\int_0^{t/2} 
  \sqrt{\|\Phi_\ep^{-1}(\cdot,s)\|_{L^1(\Omega)}E_{\pa t}(s;u)}\,
  ds.
\end{align*}
Since 
\begin{align*}
  \|\Phi^{-1}(\cdot,t)\|_{L^1(\Omega)}
  &\leq 
  \int_{\R^N}\exp\left(-\frac{A_{1\ep}}{h+2\ep}\,\frac{|x|^{2-\alpha}}{1+t}\right)\,dx
  \\
  &=
  (1+t)^{\frac{N}{2-\alpha}}
  \int_{\R^N}\exp\left(-\frac{A_{1\ep}}{h+2\ep}|y|^{2-\alpha}\right)\,dy,
\end{align*}
we deduce
\begin{align*}
\mathcal{J}_3(t)
&\leq 
C'
(1+t)^{-\frac{N-\alpha}{2(2-\alpha)}-1}
\|(u_0,u_1)\|_{H^{1}\times L^2}
\int_0^{t/2} 
  (1+s)^{
    \frac{N-\alpha}{2(2-\alpha)}-\frac{\lambda_0}{2}-\frac{1-\alpha}{2-\alpha}}
  \, ds 
\\
&\leq 
C'
\left(
    \frac{N-\alpha}{2(2-\alpha)}-\frac{\lambda_0}{2}+\frac{1}{2-\alpha}
\right)
(1+t)^{-\frac{N-\alpha}{2(2-\alpha)}-1}
  (1+t/2)^{
    \frac{N-\alpha}{2(2-\alpha)}-\frac{\lambda_0}{2}-\frac{1-\alpha}{2-\alpha}+1
  } \\
&\quad \times \|(u_0,u_1)\|_{H^{1}\times L^2}
\\
&\leq 
C''
(1+t)^{-\frac{\lambda_{0}}{2}-\frac{1-\alpha}{2-\alpha}}
\|(u_0,u_1)\|_{H^{1}\times L^2}. 
\end{align*}
Consequently, we obtain
\[
\Big\|u(\cdot,t) - e^{tL_{\ast}}[ u_0 + a(\cdot)^{-1}u_1]\Big\|_{L^2_{d\mu}}
\leq 
C'''(1+t)^{-\frac{\lambda_0}{2}-\frac{1-\alpha}{2-\alpha}}
\|(u_0,u_1)\|_{H^{2}\times H^1}.
\]

Next we show the assertion for $(u_0,u_1)$ 
satisfying $(u_0,u_1)\in (H^2\times H^1_0(\Omega))\times H^1_0(\Omega)$
(the compatibility condition of order $1$) 
via an approximation argument. 
Fix $\phi\in C_c^\infty(\R^N,[0,1])$ such that $\phi\equiv1$ on $\overline{B}(0,R_0)$ and $\phi\equiv0$ on $\R^N\setminus B(0,R_0+1)$ and define 
for $n\in\N$, 
\[
\left(\begin{array}{c}
u_{0n}
\\
u_{1n}
\end{array}\right)
=
\left(\begin{array}{c}
\phi \tilde{u}_{0n}
\\
\phi \tilde{u}_{1n}
\end{array}\right),
\quad 
\left(\begin{array}{c}
\tilde{u}_{0n}
\\
\tilde{u}_{1n}
\end{array}\right)
=
\left(1+\frac{1}{n}\mathcal{A}\right)^{-1}
\left(\begin{array}{c}
u_0
\\
u_1
\end{array}\right),
\]
where $\mathcal{A}$ is an $m$-accretive operator 
in $\mathcal{H}=H^1_0(\Omega)\times L^2(\Omega)$ 
associated with \eqref{dw}, that is, 
\[
\mathcal{A}=
\left(\begin{array}{cc}
0 & -1
\\
-\Delta & a(x)
\end{array}\right)
\]
endowed with domain 
$D(\mathcal{A})=(H^2\cap H^1_0(\Omega))\times H_0^1(\Omega)$. 
Then $(u_{0n},u_{1n})$ satisfies ${\rm supp} (u_{0n},u_{1n})\subset \overline{B}(0,R_0+1)$
and the compatibility condition of order $2$. 
Let $v_n$ be a solution of \eqref{dw} with 
$(u_{0n},u_{1n})$. 
Observe that 
\begin{align*}
\|(u_{0n},u_{1n})\|_{H^2\times H^1}^2
&\leq C^2\|\phi\|_{W^{2,\infty}}^2
\|(\tilde{u}_{0},\tilde{u}_{1})\|_{H^2\times H_1}^2
\\
&\leq C'^2\|\phi\|_{W^{2,\infty}}^2
(\|(\tilde{u}_{0},\tilde{u}_{1})\|_{\mathcal{H}}^2
+\|\mathcal{A}(\tilde{u}_{0},\tilde{u}_{1})\|_{\mathcal{H}}^2)
\\
&\leq C'^2\|\phi\|_{W^{2,\infty}}^2
(\|(u_{0},u_1)\|_{\mathcal{H}}^2
+\|\mathcal{A}(u_{0},u_{1})\|_{\mathcal{H}}^2)
\\
&\leq C''^2\|\phi\|_{W^{2,\infty}}^2
\|(u_{0},u_{1})\|_{H^2\times H^1}^2
\end{align*}
with suitable constants $C$, $C'$, $C''>0$, 
and 
\begin{gather*}
\left(\begin{array}{c}
u_{0n}
\\
u_{1n}
\end{array}\right)
\to 
\left(\begin{array}{c}
\phi u_0
\\
\phi u_1
\end{array}\right)
=
\left(\begin{array}{c}
u_0
\\
u_1
\end{array}\right)
\quad 
\text{in}\ \mathcal{H}
\end{gather*}
as $n\to \infty$ and also $u_{0n}+a^{-1}u_{1n}\to u_0+a^{-1}u_1$ 
in $L^2_{d\mu}$ as $n\to \infty$. 
Using the result of the previous step, we deduce
\[
\Big\|v_n(\cdot,t) - e^{tL_{\ast}}[ u_{0n} + a(\cdot)^{-1}u_{1n}]\Big\|_{L^2_{d\mu}}
\leq 
\tilde{C}(1+t)^{-\frac{\lambda_0}{2}-\frac{1-\alpha}{2-\alpha}}
\|(u_0,u_1)\|_{H^{2}\times H^1}
\]
with some constant $\tilde{C}>0$. 
Letting $n\to\infty$, by continuity of the $C_0$-semigroup $e^{-t\mathcal{A}}$ 
in $\mathcal{H}$
we also obtain diffusion phenomena 
for initial data in $(H^2\cap H^1_0(\Omega))\cap H^1_0(\Omega)$. 
\end{proof}

\section*{Acknowledgments}

This work is supported by 
Grant-in-Aid for JSPS Fellows 15J01600 
of Japan Society for the Promotion of Science
and 
also 
partially supported 
by Grant-in-Aid for Young Scientists Research (B), 
No.\ 16K17619. 
The authors would like to thank the referee 
for giving them valuable comments and suggestions.

\end{document}